\documentclass{amsart}
\usepackage{graphics}
\usepackage{graphicx}
\usepackage{epsfig}
\usepackage{amscd}
\usepackage{amssymb}
\usepackage{amsthm}
\usepackage{amsmath}
\numberwithin{equation}{section}
\usepackage{mathrsfs}
\usepackage{color}
\usepackage{hyperref}
\usepackage{bm}
\usepackage{float}
\usepackage{geometry}
\usepackage{tikz}
\usepackage{multirow}
\usepackage{verbatim}
\usepackage{geometry}
\usepackage{placeins}

\usetikzlibrary{matrix,arrows,arrows.meta,calc}
\usetikzlibrary{decorations.pathreplacing,decorations.markings}
\usetikzlibrary{shapes,positioning}
\tikzstyle arrowstyle=[scale=1]
\tikzstyle directed=[postaction={decorate,decoration={markings,
    mark=at position .5 with {\arrow[arrowstyle]{stealth}}}}]
\tikzstyle reverse directed=[postaction={decorate,decoration={markings,
    mark=at position .65 with {\arrowreversed[arrowstyle]{stealth};}}}]

\newtheorem{thm}{Theorem}[section]

\newtheorem{conj}[thm]{Conjecture}

\theoremstyle{plain}
\newtheorem{lem}[thm]{Lemma}

\theoremstyle{remark}

\theoremstyle{definition}

\newtheorem{ex}[thm]{Example}

\numberwithin{equation}{section}

\newcommand{\R}{\mathbb{R}}
\newcommand{\Z}{\mathbb{Z}}

\newcommand{\A}{\mathcal{A}}
\newcommand{\B}{\mathcal{B}}

\newcommand{\T}{\mathcal{T}}
\newcommand{\M}{\mathcal{M}}
\newcommand{\Ch}{\mathrm{Ch}}

\newcommand{\VG}{\mathrm{VG}}
\newcommand{\sgn}{\mathrm{sgn}}

\begin{document}

\title{Cremona invariance of filtered Varchenko--Gelfand algebras}


\author{Ye Liu}
\address{Department of Pure Mathematics, Xi'an Jiaotong-Liverpool University, Suzhou, Jiangsu 215123, P. R. China}
\email{yeliumath@gmail.com}


\date{}


\begin{abstract}
We prove that the filtered Varchenko--Gelfand algebra is invariant under a natural Cremona operation on a class of real hyperplane arrangements.  Namely, suppose that an arrangement contains all coordinate hyperplanes and that every remaining defining form is supported on two coordinates.  Swapping the two coefficients in each such form produces its Cremona transform.  Coordinatewise inversion gives a chamber bijection and an isomorphism of the corresponding filtered Varchenko--Gelfand algebras over every commutative coefficient ring.  As an application, we exhibit two arrangements of eight central planes in $\R^3$ with isomorphic filtered Varchenko--Gelfand algebras but non-isomorphic tope graphs. This disproves a conjecture of Yagi--Yoshinaga on reconstructing tope graphs from filtered Varchenko--Gelfand algebras.
\end{abstract}

\maketitle

\tableofcontents

\section{Introduction}

Let $\A=\{H_1,\ldots,H_n\}$ be a central hyperplane arrangement in a real vector space $V$.  Its complement 
\[
M(\A)=V\setminus \bigcup_{i=1}^nH_i
\]
is a disjoint union of open polyhedral cones, called \emph{chambers}. Let $\Ch(\A)$ be the set of chambers. After choosing a defining form $\alpha_i$ for each hyperplane $H_i$, every chamber determines a sign vector in $\{+,-\}^n$.  The \emph{tope graph} $\T(\A)$ has vertex set $\Ch(\A)$, with two chambers adjacent when they are separated by exactly one hyperplane.  Thus the tope graph records the codimension-one incidence of chambers.  In oriented matroid language, the chambers are the topes, and the tope graph determines the oriented matroid up to relabeling and reorientation (see \cite[Theorem~4.2.14]{Bjorner1999}).

The Varchenko--Gelfand (VG) algebra packages the same chamber set in a different way.  For a commutative ring $R$, let
\[
\VG(\A)_R=R^{\Ch(\A)}
\]
be the set of $R$-valued functions on $\Ch(\A)$ with pointwise addition and multiplication.  As an algebra it is merely a product of copies of $R$, one for each chamber.  The arrangement enters through a natural filtration introduced by Gelfand and Varchenko \cite{Gelfand1987}.  For a defining form $\alpha_i$, its Heaviside function is
\[
h_{\alpha_i}(C)=
\begin{cases}
1,&\alpha_i>0\text{ on }C,\\
0,&\alpha_i<0\text{ on }C.
\end{cases}
\]
The functions $h_{\alpha_i}$ generate $\VG(\A)_R$, and $F^p\VG(\A)_R$ is the $R$-span of products of at most $p$ Heaviside functions.  The filtration is independent of the orientations because $h_{-\alpha_i}=1-h_{\alpha_i}$.  Its associated graded algebra recovers classical combinatorial information: the graded-rank generating function is the Poincar\'e polynomial of the complexified complement.  In characteristic $2$ it is closely related to the Orlik--Solomon algebra, while in characteristic different from $2$ it can retain genuinely oriented information; see \cite{Gelfand1987,Yagi2026}.

Yagi--Yoshinaga proved that the filtered and associated graded VG algebras determine the underlying oriented matroid when $\mathrm{char}R\neq 2$ and the arrangement is generic in codimension $2$ \cite[Theorems~3.6 and~5.3]{Yagi2026}.  In the non-generic case, odd rank-two pencils can create additional degree-one idempotents, called generalized Heaviside functions.  Motivated by their analysis of these extra degree-one elements, they proposed the following reconstruction conjecture.

\begin{conj}[Yagi--Yoshinaga \cite{Yagi2026}]\label{YY}
Let $\A$ and $\B$ be real central hyperplane arrangements and let $R$ be an integral domain with $\operatorname{char}R\neq2$.  If
\[
\VG(\A)_R\cong \VG(\B)_R
\]
as filtered $R$-algebras, then
\[
\T(\A)\cong\T(\B).
\]
\end{conj}

The main result of this note provides a systematic source of filtered VG algebra isomorphisms that need not preserve tope graphs. After proving the Cremona transformation preserves the filtered VG algebras of a certain class of arrangements, we exhibit a pair of such arrangements with isomorphic filtered VG algebras but non-isomorphic tope graphs. We conclude this note with a discussion on the relation with combinatorial Cremona maps.

\section{Results}\label{cremona}

Suppose that $\A$ is an arrangement in $\R^r$ with coordinates $x_1,\ldots,x_r$ containing all the coordinate hyperplanes $x_i=0$, and that every other defining form is supported on two coordinates:
\[
\alpha_e=a_ex_{i(e)}+b_ex_{j(e)},
\qquad i(e)\neq j(e),\qquad a_eb_e\neq0.
\]
We define the \emph{Cremona transform} $\A^\kappa$ as an arrangement in $\R^r$ with coordinates $X_1,\ldots,X_r$, by retaining the coordinate hyperplanes $X_i=0$ and replacing $\alpha_e$ with
\[
\beta_e=b_eX_{i(e)}+a_eX_{j(e)}.
\]
This operation is induced on the complements by coordinatewise inversion.

\begin{thm}\label{main}
For every commutative ring $R$, coordinatewise inversion
\[
\kappa(x_1,\ldots,x_r)=(X_1,\ldots,X_r)=\left(x_1^{-1},\ldots,x_r^{-1}\right)
\]
induces an isomorphism of filtered $R$-algebras
\[
\kappa^*:\VG(\A^\kappa)_R\xrightarrow{\ \cong\ }\VG(\A)_R.
\]
Equivalently,
\[
\kappa^*F^p\VG(\A^\kappa)_R=F^p\VG(\A)_R
\qquad\text{for every }p\geq0.
\]
\end{thm}

We first record the elementary rank-two identity underlying the construction.  For a nonzero real number $u$, write $s_u=\sgn(u)\in\{\pm1\}$ and $h_u=(s_u+1)/2\in\{0,1\}$.

\begin{lem}\label{sign-identity}
If $u$, $v$, and $u+v$ are nonzero real numbers, then
\[
s_us_vs_{u+v}=s_u+s_v-s_{u+v}.
\]
\end{lem}

\begin{proof}
The sign patterns $(+1,+1,-1)$ and $(-1,-1,+1)$ cannot occur for $(s_u,s_v,s_{u+v})$.  The identity is immediate on each of the six remaining sign patterns.
\end{proof}

We now prove the main theorem.

\begin{proof}[Proof of Theorem \ref{main}]
Because $\A$ and $\A^\kappa$ contain all coordinate hyperplanes, both complements lie in the real torus $(\R^\times)^r$.  For a noncoordinate defining form of $\A$
\[
\alpha_e=a_ex_i+b_ex_j
\]
and its transformed form in $\A^\kappa$
\[
\beta_e=b_eX_i+a_eX_j,
\]
we have
\begin{equation}\label{pullback}
\kappa^*\beta_e
=\frac{b_e}{x_i}+\frac{a_e}{x_j}
=\frac{a_ex_i+b_ex_j}{x_ix_j}
=\frac{\alpha_e}{x_ix_j}.
\end{equation}
Thus $\kappa$ restricts to a homeomorphism of complements
\[
\kappa:M(\A)\xrightarrow{\ \cong\ }M(\A^\kappa)
\]
and induces a bijection of chamber sets $\kappa: \Ch(\A)\to\Ch(\A^\kappa)$.  Pullback along this chamber bijection is therefore an isomorphism of the underlying function algebras.

It remains to check the filtration.  Put
\[
\varepsilon_e=\sgn(a_e),\qquad \delta_e=\sgn(b_e).
\]
Equation \eqref{pullback} gives
\[
\kappa^*s_{\beta_e}=s_{x_i}s_{x_j}s_{\alpha_e}.
\]
Apply Lemma \ref{sign-identity} to $u=a_ex_i$ and $v=b_ex_j$.  Since
\[
s_u=\varepsilon_es_{x_i},\qquad
s_v=\delta_es_{x_j},\qquad
s_{u+v}=s_{\alpha_e},
\]
we obtain
\begin{equation}\label{sign-pullback}
\kappa^*s_{\beta_e}
=\delta_es_{x_i}+\varepsilon_es_{x_j}
-\varepsilon_e\delta_es_{\alpha_e}.
\end{equation}
Substituting $s=2h-1$ and simplifying as an identity of $\Z$-valued functions yields
\begin{equation}\label{heaviside-pullback}
\kappa^*h_{\beta_e}
=\delta_eh_{x_i}+\varepsilon_eh_{x_j}
-\varepsilon_e\delta_eh_{\alpha_e}
+\frac{(1-\varepsilon_e)(1-\delta_e)}{2}.
\end{equation}
The final constant is either $0$ or $2$, so the right-hand side is an integral affine-linear combination of Heaviside functions.  For the coordinate generators one simply has
\[
\kappa^*h_{X_i}=h_{x_i}.
\]
Consequently,
\[
\kappa^*F^1\VG(\A^\kappa)_R\subseteq F^1\VG(\A)_R
\]
for every commutative ring $R$.  Because the filtration is multiplicatively generated in degree one, it follows that
\[
\kappa^*F^p\VG(\A^\kappa)_R\subseteq F^p\VG(\A)_R
\qquad\text{for all }p\geq0.
\]
Finally, $(\A^\kappa)^\kappa=\A$ and $\kappa^2=\mathrm{id}$ on the torus.  Applying the same argument in the opposite direction gives the reverse inclusion.  Hence every filtered piece is preserved.
\end{proof}

\begin{ex}\label{counterexample}
    We define a central arrangement $\A$ in $\R^3$ by defining forms:
\[
Q_\A=xyz(x+y)(x+2y)(x+z)(y-2z)(2y-z).
\]
Then its Cremona transform $\A^\kappa$ is defined by \[
Q_{\A^\kappa}=XYZ(X+Y)(2X+Y)(X+Z)(Z-2Y)(2Z-Y).
\]

See Figure \ref{deconing} for the deconing of $\A$ and $\A^\kappa$ at $z=1$ and $Z=1$ respectively. By Theorem \ref{main}, $\A$ and $\A^\kappa$ have isomorphic filtered VG algebras. However, $\T(\A)$ has a pair of antipodal vertices of degree $5$ (one can be seen from the deconing of $\A$ as the pentagon chamber bounded by the lines $x=0,x+1=0,y-2=0,2y-1=0$ and $x+y=0$), while $\T(\A^\kappa)$ has no vertices of degree $5$. So $\T(\A)\not\cong\T(\A^\kappa)$. Hence they give a counterexample to Conjecture \ref{YY}.
\end{ex}

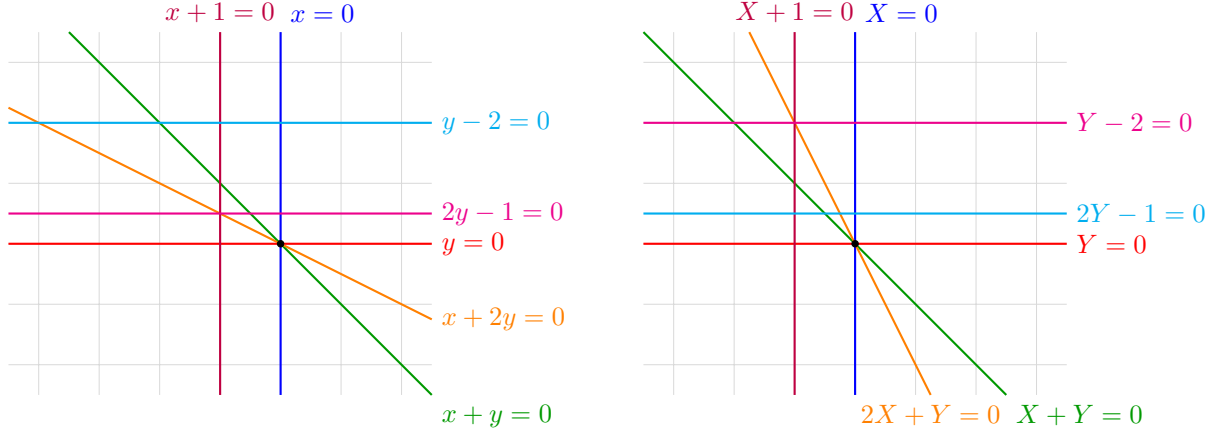
\begin{figure}
    \centering
\begin{tikzpicture}[scale=0.8]

    \begin{scope}
        \draw[step=1cm, gray!30, very thin] (-4.5, -2.5) grid (2.5, 3.5);
        
        \draw[thick, blue] (0, -2.5) -- (0, 3.5) node[above right] {$x=0$};
        \draw[thick, red] (-4.5, 0) -- (2.5, 0) node[right] {$y=0$};
        \draw[thick, green!60!black] (-3.5, 3.5) -- (2.5, -2.5) node[below right] {$x+y=0$};
        \draw[thick, orange] (-4.5, 2.25) -- (2.5, -1.25) node[right] {$x+2y=0$};
        \draw[thick, purple] (-1, -2.5) -- (-1, 3.5) node[above] {$x+1=0$};
        \draw[thick, cyan] (-4.5, 2) -- (2.5, 2) node[right] {$y-2=0$};
        \draw[thick, magenta] (-4.5, 0.5) -- (2.5, 0.5) node[right] {$2y-1=0$};

        \filldraw[black] (0,0) circle (1.5pt); 
    \end{scope}

    \begin{scope}[shift={(9.5,0)}]
        \draw[step=1cm, gray!30, very thin] (-3.5, -2.5) grid (3.5, 3.5);
        
        \draw[thick, blue] (0, -2.5) -- (0, 3.5) node[above right] {$X=0$};
        \draw[thick, red] (-3.5, 0) -- (3.5, 0) node[right] {$Y=0$};
        \draw[thick, green!60!black] (-3.5, 3.5) -- (2.5, -2.5) node[below right] {$X+Y=0$};
        \draw[thick, orange] (-1.75, 3.5) -- (1.25, -2.5) node[below] {$2X+Y=0$};
        \draw[thick, purple] (-1, -2.5) -- (-1, 3.5) node[above] {$X+1=0$};
        \draw[thick, magenta] (-3.5, 2) -- (3.5, 2) node[right] {$Y-2=0$};
        \draw[thick, cyan] (-3.5, 0.5) -- (3.5, 0.5) node[right] {$2Y-1=0$};

        \filldraw[black] (0,0) circle (1.5pt); 
    \end{scope}

\end{tikzpicture}
    
    \caption{Left: deconing of $\A$ at $z=1$; Right: deconing of $\A^\kappa$ at $Z=1$.}
    \label{deconing}
\end{figure}

\section{Relation with combinatorial Cremona maps}

The construction in Section \ref{cremona} is the realizable, real-sign counterpart of the combinatorial Cremona maps introduced by Shaw--Werner \cite{Shaw2023} and further studied by Rettenmayr--Werner \cite{Rettenmayr2025}.  Let $\M=\M(\A)$ be the matroid represented by the defining forms of $\A$, and let
\[
b=\{x_1,\ldots,x_r\}
\]
be the basis corresponding to the coordinate hyperplanes.  For $i<j$, set
\[
F_{ij}=\operatorname{cl}_\M\{x_i,x_j\}\setminus\{x_i,x_j\}.
\]
The hypothesis of Theorem \ref{main} says precisely that the sets $F_{ij}$ partition $E(\M)\setminus b$, that is, a nonbasis defining form belongs to $F_{ij}$ exactly when it is supported on the two coordinates $x_i$ and $x_j$.  When $\M$ is connected, Shaw--Werner's criterion \cite[Theorem~8.3]{Shaw2023} says that this is exactly the condition under which $b$ defines a combinatorial Cremona automorphism of the coarse Bergman fan $B_c(\M)$.  Such a basis is called a \emph{Cremona basis} in \cite{Rettenmayr2025}.

In the realizable setting, the two constructions arise from the same standard Cremona transformation.  Indeed, for $e\in F_{ij}$, Equation \eqref{pullback} reads
\[
\kappa^*\beta_e=\frac{\alpha_e}{x_ix_j}.
\]
Tropicalizing this monomial relation gives the combinatorial Cremona map on the Bergman fan.  Taking real signs instead, and applying Lemma \ref{sign-identity}, gives the affine-linear Heaviside identity \eqref{heaviside-pullback}, which is the additional ingredient that proves preservation of the VG filtration.  Thus the work of Shaw--Werner records the tropical, unoriented invariance of this Cremona move, while Theorem \ref{main} records its effect on real chamber functions.

Rettenmayr--Werner \cite{Rettenmayr2025} introduced support graphs for matroids with a Cremona basis and used them to study the structure of such matroids, the interaction of distinct Cremona bases, realizability questions, and coarse Bergman fan automorphisms of Coxeter arrangement matroids.  In Example \ref{counterexample}, with $b=\{x,y,z\}$, one has
\[
F_{xy}=\{x+y,x+2y\},\qquad
F_{xz}=\{x+z\},\qquad
F_{yz}=\{y-2z,2y-z\}.
\]
Hence their support graph is a triangle with edge multiplicities $2,1,2$.  Their theory concerns the underlying unoriented matroid, whereas the tope graph depends on its orientation.  Our counterexample shows that the same Cremona move can preserve both the coarse Bergman fan and the filtered VG algebra while changing the tope graph.

\section*{Acknowledgment}

The author thanks Masahiko Yoshinaga for his valuable comments on this paper.

\renewcommand\refname{References}
\bibliographystyle{hep}
\bibliography{VGvsTope}

\end{document}